\documentclass[11pt]{article}
\usepackage{authblk}
\usepackage[top = 1in, bottom = 1in, left = 1in, right = 1in]{geometry}
\usepackage{graphicx}
\usepackage{latexsym}
\usepackage{amsthm, amsfonts,amsmath,amssymb}
\usepackage[svgnames]{xcolor} 

\usepackage{pgf,tikz}
\usepackage{tkz-graph}
\usepackage{enumitem}
\usepackage{verbatim}
\usepackage{hyperref}
\hypersetup{
    colorlinks=true, 
    linktoc=all,     
    linkcolor=blue,  
    citecolor=ForestGreen,
}
\newcommand{\red}[1]{{\leavevmode\color{red}#1}}

\newcommand{\brown}[1]{{\leavevmode\color{Brown}#1}}
\newcommand{\magenta}[1]{{\leavevmode\color{Magenta}#1}}

\newtheorem{theorem}{Theorem}[section]

\newtheorem{lemma}[theorem]{Lemma}
\newtheorem{definition}[theorem]{Definition}
\newtheorem{corollary}[theorem]{Corollary}

\theoremstyle{remark}

\newtheorem{Remark}[theorem]{\bf Remark}
\newtheorem*{Example}{\bf Example}
\newtheorem*{Claim}{\it Claim}

\newtheorem*{claim}{Claim}

\setlist[enumerate]{parsep=0pt, itemsep=0pt, topsep=0pt}
\setlist[itemize]{parsep=0pt, itemsep=0pt, topsep=0pt}
\numberwithin{equation}{section}

\def\eps{\varepsilon}
\def\la{\lambda}
\def\a{\alpha}
\def\be{\beta}
\def\k{\kappa}

\def\ga{\gamma}
\def\part{\partial}
\def\Cal{\mathcal}

\newcommand{\pr}{\operatorname{\mathbb P}}
\newcommand{\mean}{\operatorname{\mathbb{E}}}

\newcommand{\End}{\operatorname{END}}
\renewcommand{\cal}{\mathcal}
\newcommand{\floor}[1]{\lfloor #1 \rfloor}

\newcommand{\bs}{\boldsymbol}
\newcommand{\lp}{\left(}
\newcommand{\rp}{\right)}

\newcommand{\brac}[1]{\left(#1\right)}
\newcommand{\bfrac}[2]{\brac{\frac{#1}{#2}}}

\newcommand{\beq}{\begin{equation}}
\newcommand{\eeq}{\end{equation}}
\newcommand{\gnk}{G_{n,k}}

\newcommand{\ug}[1][n,k]{G_{#1}}

\title{Perfect matchings and Hamilton cycles in uniform attachment graphs}
\date{}
\author{H\"useyin Acan}
\affil{Department of Mathematics, Drexel University, Philadelphia, PA 19104\\
\texttt{huseyin.acan@drexel.edu}}

\begin{document}
\maketitle

\begin{abstract}
We study Hamilton cycles and perfect matchings in a uniform attachment graph. In this random graph, vertices are added sequentially, and when a vertex $t$ is created, it makes $k$ independent and uniform choices from $\{1,\dots,t-1\}$ and attaches itself to these vertices. Improving the results of Frieze, P\'erez-Gim\'enez, Pra\l{}at and  Reiniger (2019), we show that, with probability approaching 1 as $n$ tends to infinity, a uniform attachment graph on $n$ vertices has a perfect matching for $k \ge 5$ and a Hamilton cycle for $k\ge 13$. One of the ingredients in our proofs is the identification of a subset of vertices that is least likely to expand, which provides us with better expansion rates than the existing ones.

\vspace{.5cm}

\end{abstract}

\section{Introduction}\label{sec:intro}

A \emph{uniform attachment graph} (UAG) is a dynamic random graph where the vertices arrive one at a time. Upon their arrival, each vertex makes $k$ choices -- each one uniform and independent of the others -- from the previous vertices, and attaches itself to these vertices. Naturally this process produces a directed multigraph at every stage, but in this paper we disregard the edge orientations and multiple edges and view it as a simple undirected graph. We study Hamilton cycles and perfect matchings in a UAG. More specifically, we are concerned with identifying those $k$'s (the parameter of the graph) for which the graph is likely/unlikely to contain a Hamilton cycle or a perfect matching.
 
As far as the author is aware, this graph first appeared in a paper by Bollob\'as, Riordan, Spencer, and G. Tusn\'ady~\cite{BRST}, who noted that the graph in question has geometric degree distribution. More recently, it has appeared in several more works~\cite{MJKS,FGPR,AP}. In~\cite{MJKS}, Magner, Janson,  Kollias, and Szpankowski studied the symmetry in a UAG and showed that, for $k=1$, the graph has a nontrivial automorphism with high probability (whp), and for $k=2$, this probability is bounded away from 0. They also conjectured that whp there is no nontrivial automorphism for $k\ge 3$.
Frieze, P\'erez-Gim\'enez, Pra\l{}at and  Reiniger~\cite{FGPR} studied the Hamilton cycles and perfect matchings in a UAG, and in this paper we improve their results. Most recently, bootstrap percolation on uniform attachment graphs is studied in~\cite{AP}.

In the studies~\cite{BRST,MJKS,FGPR}, uniform attachment graphs appear together with preferential attachment graphs introduced by Barab\'asi and Albert~\cite{BA} and made rigorous by Bollob\'as and Riordan~\cite{BR}. 
In fact, the uniform attachment graph is the limiting distribution of a generalized Barab\'asi-Albert graph model as the additional parameter $\delta$ tends to infinity. (See van der Hofstad~\cite[Chapter 8]{Hofstad} for a description of this generalized model.)

Thresholds for Hamilton cycles and perfect matchings have been studied widely for various random graph models, starting with a paper of Erd\H{os} and R\'enyi~\cite{ER60}, where they posed the question of finding the threshold for the model  $G_{n,m}$. The same authors~\cite{ER66} later proved that $G_{n,m}$ has a perfect matching whp when $m=0.5n\log n+\omega n$ for any $\omega\to \infty$. 
As noted by the authors, this is also the number of necessary edges for the disappearance of the last isolated vertex. Later improvements were given by Bollob\'as and Thomason~\cite{BT} as well as Bollob\'as and Frieze~\cite{BF}. In the case of Hamilton cycles, improving the results of P\'osa~\cite{Posa} and Kor\v{s}unov~\cite{Korsunov}, a detailed picture for Hamilton cycles is given by Koml\'os and Szemer\'edi~\cite{KS}. Bollob\'as~\cite{Bollobas} and  Ajtai, J. Koml\'os and E. Szemer\'edi~\cite{AKS} proved that whp the graph becomes Hamiltonian as soon as the minimum degree becomes 2.

As the main obstacle for the existence of a perfect matching (resp.\ Hamilton cycle) is the presence of degree-0 (resp.\ degree at most 1) vertices in $G_{n,m}$ (and a closely related model $G_{n,p}$), it is natural to consider random graphs with a given minimum degree. In that case, one can hope to find a perfect matching or a Hamilton cycle in a much sparser graph, maybe with a constant edge density. In fact, this has been the case for a variety of graph models.

For a $d$-regular random graph, Bollob\'as~\cite{Bollobas83} and Fenner and Frieze~\cite{FF84} showed that the graph is Hamiltonian for sufficiently large constant $d$. Later, Robinson and Wormald~\cite{RW} proved that it is Hamiltonian for every fixed $d\ge 3$. This is the best one can expect since a 2-regular graph is a union of cycles and whp a 2-regular graph has more than one cycle. In the case of perfect matchings, when the order is even, a random $d$-regular graph has a perfect matching whp if $d=1$ or $d\ge 3$ and does not have a perfect matching whp if $d=2$, see Bollob\'as~\cite[Corollary 7.33]{Bollobas}.

A model similar to the uniform attachment model is the $k$-out model, where the vertex set is $[n]$ and each vertex $t\in [n]$ makes $k$ independent and uniform choices from $[n]$ (as opposed to the vertices in $[t-]$).
For $k$-out random graphs, after a sequence of improvements in~\cite{FF83,Frieze88,FL90,CF00}, Bohman and Frieze~\cite{BohmanFrieze} proved that $3$-out is Hamiltonian, which is the best possible. For the perfect matching problem, Frieze~\cite{Frieze86} showed that $2$-out has a perfect matching whp as long as the order of the graph is even.
For a bibliography on Hamilton cycles in various other graph models, we refer the reader to the survey by Frieze~\cite{Frieze19}.

One of our motivations for this work comes from a paper of Frieze, P\'erez-Gim\'enez, Pra\l{}at and  Reiniger~\cite{FGPR}, where the authors studied the Hamiltonicity and perfect matchings in uniform attachment graphs and preferential attachment graphs. 
They proved that whp a random uniform attachment graph has a perfect matching for $k\ge 159$ and it is Hamiltonian for $k\ge 3214$. For the preferential attachment graph, they showed that the graph has a perfact matching for $k\ge 1253$ and is Hamiltonian for $k\ge 29500$. In this paper, we improve the results for the uniform attachment graph model and show that the graph has a perfect matching for $k\ge 5$ and is Hamiltonian for $k\ge 13$. As another motivation, we are  curious to understand how different this model is from the similarly-defined $k$-out model and the preferential attachment model.

We introduce our notation and give our results in the next section. In Section~\ref{sec:expansion}, we study the expansion of $\gnk$. 
In Sections~\ref{sec:matchings} and \ref{sec:cycles}, we prove our results for perfect matchings and Hamilton cycles, respectively.
Our proofs are based on some of the techniques developed in~~\cite{FGPR}. The improvement of the results follows from the expansion rates established in Section~\ref{sec:expansion} and a slight modification of their algorithms that allows us to start with better initial partial matchings and long paths. We believe the results in Section~\ref{sec:expansion} are interesting in their own right.

\section{Main results}\label{sec:results}
Here we give the main results of the paper following the notation and terminology we will use.

\medskip
\textbf{Notation and terminology.} Most of our notation and terminology is standard but we will note it for completeness.
We say that an event $E_n$, parametrized by $n$, occurs \emph{with high probability} (whp) if $\lim_{n\to \infty}\pr(E_n)=1$.
A matching $M$ \emph{isolates} a vertex $v$ if $v$ is not incident to any of the edges in $M$. To accommodate odd orders, we define a \emph{perfect matching} as a matching that isolates at most one vertex. A {Hamilton cycle} is a cycle that contains all the vertices of the graph.  For two integers $a$ and $b$, we denote by $[a,b]$ the set $\{a,a+1,\dots,b\}$. We simply write $[n]$ for $[1,n]$. We will denote by $\gnk$ the uniform attachment graph with $n$ vertices, where each vertex $j\in [2,n]$ makes $k$ choices uniformly and independently from the vertex set $[j-1]$ and attaches itself to these vertices.

For a graph $G=(V,E)$ and $X\subseteq V$, we denote by $N_G(X)$ the out-neighbors of the vertices in $X$, that is,
\[
N_G(X):=\{u \in V\setminus X\, : \, xu \in E \text{ for some } x\in X\}.
\]
(We will frequently supress the subscript $G$ when the graph is evident from the context.)

For a given graph $G=(V,E)$, in Section~\ref{sec:matchings}, we will say that $X\subseteq V$ \emph{expands} if $|N(X)| \ge |X|$. Similarly,  in Section~\ref{sec:cycles}, we will say that $X$ \emph{expands well} if $|N(X)| \ge 2|X|$.

We are now ready to give our main results.

\begin{theorem}\label{thm:k=4}
Whp, $\ug[n,4]$ has a matching that contains all but a bounded number of vertices.
\end{theorem}

\begin{theorem}\label{thm:k=5}
Whp, $\ug[n,5]$ has a perfect matching.
\end{theorem}

\begin{theorem}\label{thm:k=13}
Whp, $\ug[n,13]$ is Hamiltonian.
\end{theorem}

\begin{Remark}
Since $\gnk$ is the union of $k$ independent $\ug[n,1]$'s, the probability that $\gnk$ has a perfect matching (Hamilton cycle) increases with $k$. Frieze et al.~\cite{FGPR} showed that, whp, $\ug[n,2]$ is not Hamiltonian. On the other hand, $\ug[n,1]$ is a tree, and it is easy to show that it has many induced cherries whp, which implies the lack of a perfect matching in $\ug[n,1]$. So with the resuls above, we have
\[
\min\{k: \gnk \text{ has a perfect matching whp}\} \in \{2,3,4,5\}
\]
and
\[
\min\{k: \gnk \text{ has a Hamilton cycle whp}\}  \in \{3,\dots,13\}.
\]
\end{Remark}

\section{Expansion of $\gnk$}\label{sec:expansion}

The proofs of the main theorems rely on the fact that small sets in $\gnk$ expands well, and this section is dedicated to establishing this fact.


\begin{definition}
Let $G=(V,E)$ be a graph and $\alpha\in (0,1)$ and $\beta >0$ be two constants. We say that $G$ is an $(\alpha,\beta)$-expander if 
$|N_G(X)| \ge \beta|X|$ for every $X\subseteq V$ such that $|X|\le \a|V|$.
\end{definition}

To show the existence of perfect matchings and Hamilton cycles in uniform attachment graphs, we will use that $\ug$ is $(\a_1(k),1)$-expander and $(\a_2(k),2)$-expander whp, respectively, for suitable constants $\a_1(k)$ and $\a_2(k)$. As one would expect, large $\a_i$ helps us.

In order to show $\gnk$ is an $(\a,\be)$ expander, we first identify a subset $X\subseteq V$ which maximizes the probability
$\pr(|N(X)| <\beta |X|)$
for a given subset cardinality $|X|$.
For $\ell \in [n]$, let $\preceq_\ell$ denote the \emph{lexicographic order} on the $\ell$-subsets of $[n]$; if $X=\{x_1<\cdots<x_\ell\}$ and $Y=\{y_1<\cdots <y_\ell\}$, then $X \preceq_\ell Y$ if and only if $x_i\le y_i$ for all $i\in [\ell]$.

\begin{lemma}\label{lem:sd}
If  $X\preceq_\ell Y$, then $|N_{\ug}(X)|$ stochastically dominates $|N_{\ug}(Y)|$.
\end{lemma}

The following corollary follows immediately.

\begin{corollary}\label{cor:worstX}
For any positive integers $\ell$ and $m$, the maximum of the set of probabilities
\[
\{\pr(|N_{\ug}(X)| <m) \, : \,  X\subset [n] \text{ and } |X|=\ell\}
\]
is attained at $X=[n-\ell+1,n]$. \qed
\end{corollary}

\begin{proof}[{Proof of Lemma~\ref{lem:sd}}]
It is enough to prove that the statement of the lemma is true for all pairs $(X,X')$, where $X'$ covers $X$, i.e.\ there exists some $x$ such that
$X\setminus X'=\{x\}$ and $X'\setminus X=\{x+1\}$.
Suppose $(X,X')$ is such a pair.

Let us simply write $G$ for the graph $\ug$. 
Note that $G$ is determined by $k(n-1)$ choices made by the vertices $[2,n]$.
Let $\bs Z=(Z_{2,1},\dots,Z_{2,k},\dots, Z_{n,1}\dots, Z_{n,k})$,
where $Z_{i,j}$ represents the $j$-th choice of vertex $i$. 
Hence $Z_{i,j} \in [i-1]$ for all $i\in [n]$ and $j\in[k]$.
Let $\cal S$ denote the set of all $(n-1)!^k$ possible sequences of choices.
For $\bs z\in \cal S$, let $\Gamma(\bs z)$ denote the resulting graph, i.e.\ for $u<v$, the edge $uv$ is present if and only if $u\in \{z_{v-1,1},\dots,z_{v-1,k}\}$.
Since $\pr(\bs Z=\bs z) = 1/(n-1)!^k$ for every $\bs z \in \cal S$, for a given graph $H$ on $[n]$,
\[
\pr(G=H)= \pr(\Gamma(\bs Z)=H)=\frac{\#\{\bs z : \Gamma(\bs z)=H\}}{(n-1)!^k}.
\]
Our plan is to find a bijection $\bs z \mapsto \bs z'$ such that $|N_{\Gamma(\bs z')}(X')| \le |N_{\Gamma(\bs z)}(X)|$.
For a given sequence $\bs z \in \cal S$, let us call $(z_{j,1},\dots,z_{j,k})$ the $j$-th block for $j \in [2,n]$.
For $\bs z \in \cal S$, we obtain $\bs z'$ from $\bs z$ in the following way.
\begin{itemize}
\item For any $i<x$, the $i$-th blocks of $\bs z$ and $\bs z'$ are the same. In other words, the first $(x-2)k$ entries of $\bs z'$ match those of $\bs z$.
\item We swap the corresponding entries of the $x$-th and $(x+1)$-th blocks as long as there is no $x$ in the $(x+1)$-th block of $\bs z$. If there are some $x$'s in the $(x+1)$-th block of $\bs z$, then those components are not swapped but the other ones are swapped.
\item For the entries appearing after $(x+1)$-th block, we make the change $x \longleftrightarrow x+1$. In other words, we replace every occurrence of $x$ in $\bs z$ with an $x+1$,  and every occurrence of $x+1$ with an $x$. The other entries are not affected. For example, if $\bs z$ after the $(x+1)$-th block looks like $(\dots, x, \dots, x, \dots, x+1,\dots x)$, then $\bs z'$ after the $(x+1)$-th block looks like $(\dots, x+1, \dots, x+1, \dots, x,\dots x+1)$, where the dotted parts are the same.
\end{itemize}
Formally, $\bs z'$ is defined as
\beq\label{zz'}
z'_{i,j}=\begin{cases}
z_{i+1,j}        & \text{if $i=x$ and $z_{x+1,j}\not =x$}\\
z_{i-1,j}        & \text{if $i=x+1$ and $z_{x+1,j}\not =x$}\\
x & \text{if $i>x+1$ and $z_{i,j}=x+1$}\\
x+1 & \text{if $i>x+1$ and $z_{i,j}=x$}\\
z_{i,j} & \text{otherwise.}
\end{cases}
\eeq
\begin{Example}
Let $k=2$, $n=8$, and $x=4$. For convenience, entries in the same block are put between brackets.
\[
\bs z= ([1,1], [2,1] ,[3,1],[4,2],[3,5],[4,4], [5,2]) 	\mapsto
\bs z'=(\underbrace{[1,1], [2,1]}_{\text{no change}} , \underbrace{[3,2],[4,1]}_{\text{partial swap}}, \underbrace{[3,4],[5,5], [4,2]}_{x\longleftrightarrow x+1})
\]
\end{Example}
Note that the procedure described above to get $\bs z'$ from $\bs z$ is reversible, and hence the function $\bs z\mapsto \bs z'$ is a bijection.
\begin{Claim}
$|N_{\Gamma(\bs z')}(X')| \le |N_{\Gamma(\bs z)}(X)|$.
\end{Claim}

\begin{proof}[Proof of the claim]
Let $A=N_{\Gamma(\bs z)}(X)$ and $A'=N_{\Gamma(\bs z')}(X')$. We will prove
\begin{enumerate}
\item[(1)] $x\in A' \iff x+1\in A$ and 
\item[(2)] $A' \setminus\{x\} \subseteq A\setminus \{x+1\}$,
\end{enumerate}
from which the proof of the lemma follows easily.

Let us prove (1) first. Note that if $z_{x+1,j} =x$ for some $j \in [k]$, then $x\in A'$ and $x+1\in A$, in which case we are done.
Now suppose $z_{x+1,j}\not=x$ for all $j\in[k]$. By~\eqref{zz'},
\begin{align*}
x\in A' &\iff z_{x,j} \in X' \text{ for some } j \in [k], \text{ or } z_{i,j}=x \text{ for some } i\in X'\setminus \{x+1\} \text{ and } j\in [k] \\
&\iff z_{x+1,j} \in X \text{ for some } j\in [k], \text{ or } z_{i,j}=x+1 \text{ for some } i\in X \text{ and } j\in [k] \\
&\iff x+1\in A.
\end{align*}
This proves (1).
Now suppose $y\in A'\setminus\{x\}$. (In particular, this means $y\not=x$ and $y\not= x+1$.) Hence, either $z'_{y,j} \in X'$ for some $j\in[k]$, or $z'_{i,j} =y$ for some $i\in [y+1,n]$ and $j\in [k]$. 
By~\eqref{zz'},
\begin{itemize}
\item if $z'_{y,j} \in X'$, then $z_{y,j} \in X$,
\item if $z'_{i,j} =y$ for some $i\in X\setminus\{x+1\}$, then $z'_{i,j} =y$,
\item if $z'_{x+1,j}=y$, then $z_{x,j}=y$.
\end{itemize}
In any case, $y\in A$, which finishes the proof of (2) and the proof of the claim.
\end{proof}
Finally,  by the claim above and since $\bs z\mapsto \bs z'$ is a bijection, we have
\begin{align}\label{NGX vs NGX'}
\pr(|N_G(X)| \ge m) &=\pr(|N_{\Gamma(\bs Z)}(X)| \ge m) \notag \\
&= \frac{ \#\{\bs z \in {\cal S}: |N_{\Gamma(\bs z)}(X)| \ge m \}}{(n-1)!^k} \ge   \frac{ \#\{\bs z \in {\cal S}: |N_{\Gamma(\bs z')}(X')|\ge m \}}{(n-1)!^k} \\
&= \pr(|N_{\Gamma(\bs Z)}(X')| \ge m) =\pr(|N_G(X')| \ge m). \notag \tag*{\qedhere}
\end{align}
\end{proof}

We are particularly interested in what Corollary~\ref{cor:worstX} gives us in the cases of $m=\ell$ and $m=2\ell$, which will be needed for finding perfect matchigs and Hamilton cycles, respectively.
Now let us bound $\pr(|N_{\gnk}(X)|<m)$ and  $\pr(|N_{\gnk}(X)|<2m)$, where $X=[n-m+1,n]$.

\begin{lemma}\label{lem:tailX}
Let $X=[n-m+1,n]$. We have
\begin{align} 
\pr(|N_{\gnk}(X)|<m) & \le {n-m \choose m-1} \lp \frac{(2m)_m}{(n)_m}\rp^k, 			\label{<m}		\\
\pr(|N_{\gnk}(X)|<2m) & \le {n-m \choose 2m-1} \lp \frac{(3m)_m}{(n)_m}\rp^k,		 	\label{<2m}
\end{align}
where $(a)_j:=a(a-1)\cdots(a-j+1)$ for any positive integer $j$.
\end{lemma}

\begin{proof}
We will only prove~\eqref{<m} since the proof of~\eqref{<2m} is almost identical.
For simplicity, write $N(X)$ for $N_{\gnk}(X)$.
We have
\beq\label{sumoverY}
\pr(|N(X)|<m) \le \sum_{Y\subseteq [n-m]\atop |Y|=m-1} \pr(N(X)\subseteq Y).
\eeq
On the other hand, for a given $Y$ of size $m-1$,
\[
\pr(N(X)\subseteq Y) = \prod_{i=0}^{m-1} \lp \frac{m-1+i}{n-m+i}\rp^k
\]
since for $N(X)\subseteq Y$ to happen, vertex $n-m+1+i$ must choose from $Y\cup [n-m+1,n-m+i]$. It is easy to see that the right hand side is smaller than $((2m)_m/(n)_m)^{k}$. Since there are ${n-m \choose m-1}$ summands in the sum on the right hand side of~\eqref{sumoverY}, we have
\[
\pr(|N_G(X)|<m) \le {n-m \choose m-1} \lp \frac{(2m)_m}{(n)_m}\rp^k .							 \qedhere
\]
\end{proof}

\begin{corollary}\label{a1,a2 expanders}
Let 
\[
H(x)=-x\log_2x-(1-x)\log_2(1-x)
\]
be the binary entropy function. 
For $k\ge 3$, let $\a_1(k)$ be the unique solution \textup{(}in $(0,1/2)$\textup{)} of 
\[
2(k+1)x+H(2x)-kH(x)=0.
\]
Similarly, for $k\ge 4$, let $\a_2(k)$ be the unique solution \textup{(}in $(0,1/3)$\textup{)} of 
\[
\log_2(27/4)(k+1)x + H(3x) -kH(x)=0.
\]
Then the following hold:
\begin{enumerate}
\item[(i)] For any $\a<\a_1(k)$,
\beq\label{expansion1}
\sum_{|X|\le \a n}\pr(|N_{\ug}(X)|< |X|) = O\left(n^{3-2k}\right).
\eeq
\item[(ii)]
For any $\a<\a_2(k)$,
\beq\label{expansion2}
\sum_{|X|\le \a n}\pr(|N_{\ug}(X)|< 2|X|) = O\left( n^{2-k}\right).
\eeq
\end{enumerate}
\end{corollary}

\begin{proof}
The proof of~\eqref{expansion2} is almost identical to the proof of~\eqref{expansion1}, so we will prove~\eqref{expansion1} and note the difference for the other equation. Let $\a<\a_1(k)$ and $M:=\floor{\a n}$.
By Corollary~\ref{cor:worstX} and Lemma~\ref{lem:tailX}, the probability that there is some subset $X$ such that $|X|=m>|N(X)|$ is at most
\[
q_m:= {n\choose m} {n-m \choose m-1} \lp \frac{(2m)_m}{(n)_m}\rp^k = {n\choose 2m-1} {2m-1 \choose m} \lp \frac{{2m \choose m}}{{n \choose m}}\rp^k.
\]
Hence the probability that there is some such $X$ with $|X|\le M:=\a n$ is at most
\[
 \sum_{m=2}^M q_m =\sum_{m=2}^M {n\choose 2m-1} {2m-1 \choose m} \lp \frac{{2m \choose m}}{{n \choose m}}\rp^k.
\]
Since
\[
\frac{q_{m+1}}{q_m}= 2^k\cdot \frac{(n-2m+1)(n-2m)}{m(m+1)} \cdot \lp \frac{2m+1}{n-m}\rp^k,
\]
$q_m$ is decreasing quickly for $m\le \eps n$, where $\eps$ is a small enough constant. 
Also, since $q_2=O(n^{3-2k})$ and $q_2=O(n^{5-3k})$ we have
\[
\sum_{m=2}^M q_m= \sum_{m=2}^{\eps n} q_m+ \sum_{m=\eps n+1}^M q_m=\le O(n^{3-2k})+n \max_{\eps n\le m\le M}q_m.
\]
For $m=cn>\eps n$, using
\[
{2m-1 \choose m} \le {2m \choose m} \le 2^{2m}=2^{2cn}, \quad {n\choose 2m-1} \le 2^{nH(2c)}, \quad {n\choose m}^{-1} \le O\left( \sqrt n \,2^{-nH(c)}\right),
\]
we get
\[
q_m \le O\left( n^{k/2} \cdot 2^{n\big(2c(k+1)+H(2c)-kH(c)\big)}\right).
\]
Since $f(x)=2(k+1)x+H(2x)-kH(x)$ is negative on $(0,\a_1(k))$, the right hand side of the inequality above is exponentially small, which means that $n \max\{q_m : \eps n\le m\le M\}$
is exponentially small.

Equation~\eqref{expansion2} can be obtained by running the same argument with
\[
q'_m:=  {n\choose m} {n-m \choose 2m-1} \lp \frac{(3m)_m}{(n)_m}\rp^k
\]
instead of $q_m$. (In this case the sum starts with $m=1$ and $q'_1= O(n^{2-k})$.)
\end{proof}

\subsection*{Numerical results for $\a_1(k)$ and $\a_2(k)$}
Recall from Corollary~\ref{a1,a2 expanders} that $\alpha_1(k)$ (for $k\ge 3$) and $\alpha_2(k)$ (for $k\ge 4$) are the unique positive solutions of the equations
\[
2(k+1)x+H(2x)-kH(x)=0 \quad \text{and} \quad \log_2(27/4) (k+1)x+H(3x)-kH(x)=0,
\]
respectively, and $\gnk$ is $(c,1)$-expander for any $c<\alpha_1(k)$ and $(d,2)$-expander for any $d<\a_2(k)$.
Numerical computations performed in MATLAB gives
\beq\label{a1 values}
\alpha_1(3)> 0.043, \quad \alpha_1(4)> 0.172,         
\eeq
and
\begin{alignat}{5}\label{a2 values}
&\alpha_2(4) > 0.005, \quad  &&\alpha_2(5)> 0.048,  \quad  &&\alpha_2(6) >0.101,   \quad   &&\alpha_2(7) > 0.144,   \quad   &&\alpha_2(8) > 0.177,  \notag \\
&\alpha_2(9) > 0.202,  \quad &&\alpha_2(10) > 0.221,   \quad &&\alpha_2(11) > 0.235, \quad  &&\alpha_2(12) > 0.247,   \quad &&\alpha_2(13) > 0.257.
\end{alignat}
(The right hand sides of the inequalities above match $\a_i(j)$ up to three digits after the decimal points,  so they can be used as approximate values of $\a_i(j)$.)

\section{Perfect matchings}\label{sec:matchings}

For a graph $G$ without a perfect matching, let $A=A(G)$ denote the set of vertices that are not covered by at least one maximum matching. 
Also, for $v\in A$, let $B(v)$ denote the set of vertices $w\not=v$ for which there is a maximum matching that does not cover both $v$ and $w$. 
It follows from the definition that $B(v)\subseteq A$ for every $v\in A$.
The following lemma is one of the key tools for the proof of Theorems~\ref{thm:k=4} and \ref{thm:k=5}.

\begin{lemma}[{\cite[Lemma 6.3]{FK}}]\label{B and N(B)}
If $G$ is a graph without a perfect matching and $v\in A(G)$, then
\[
|N(B(v))| <|B(v)|.
\]
\end{lemma}

This lemma tells us that $B(v)$ does not expand. 
Since, by~\eqref{expansion1}, all the sets of size at most $(\a_1(k)-\eps)n$ expand in $\gnk$ for any constant $\eps>0$,
$B(v)$ must be large (whp) for every $v\in A(\gnk)$ as long as $\gnk$ does not have a perfect matching.

As noted in~\cite{FGPR}, for $k=k_1+k_2$, we view $G_{n,k}$ as the union of two independent graphs $G_{n,k_1}$ and $G_{n,k_2}$. 
The following lemma (and its proof) is essentially from~\cite{FGPR}, tailored for our purposes.

\begin{lemma} \label{thm: alpha and gamma}
Let $\gamma$ and $\alpha$ be two positive constants. 
Suppose $G$ is a graph on the vertex set $[n]$, which has a matching that isolates at most $\gamma n$ vertices.
Suppose also that every vertex subset of size at most $\alpha n$ expands in $G$.
Let $k$ be a positive  integer and 
\[
\zeta:=\alpha- \frac{1}{k+1}+\frac{(1-\alpha)^{k+1}}{k+1}.
\]
\begin{enumerate}
\item[(i)] If $\zeta> \gamma/2$, then $G \cup \ug$ has a perfect matching whp.
\item[(ii)] If $\zeta < \gamma/2$,  then $G \cup \ug$ has a matching that isolates at most $(1+o(1))(\gamma-2\zeta)n$ vertices whp.
\end{enumerate}
\end{lemma}

In the proof of this lemma and in several other places we will use the following Chernoff bound (see~\cite[Chapter 2]{JLR}).

\begin{theorem}
If $X_1,\dots,X_n$ are independent Bernoulli random variables, $X=\sum_{i=1}^n X_i$, and $\mu=\mean[X]$, then
\beq\label{Chernoff}
\pr(|X-\mu| \ge \eps\mu)\le 2e^{-\eps^2\mu/3}
\eeq
for any $0<\eps<3/2$. In particular, the same bound holds when $X$ is a binomial random variable with $\mu=\mean{X}$.
\end{theorem}

\begin{proof}[Proof of Lemma~\ref{thm: alpha and gamma}]
Let $A$ and $B(v)$ be as defined above.
We start with $G$ and expose the edges of $G_{n,k}$, one vertex at a time in some particular order, and add them to the current graph.
So we have $G_0=G$ and $G_i=G_{i-1}\cup E_i$ for $i\ge 1$, where $E_i$ is the set of edges exposed in the $i$-th step.
In this process, in step $i$, we expose the edges emanating from the largest unexposed element of $A(G_{i-1})$. (If no such vertex exists, then $G_{i-1}$ has a perfect matching and we are done.)
Let $v_1$ be the largest element of $A(G_0)$. 
Hence $E_1$ is the set of edges emanating from $v_1$. 
If one of these edges joins $v_1$ to a vertex in $B(v_1)$, we can improve the maximum matching size by 1 by adding this new edge to a maximum matching that does not contain $v_1$.
In any case we update $A$ according to $G_1$. In the next step, we expose the edges emanating from the largest unused element of $A(G_1)$, and so on.

Note that, since $G \subseteq G_i$, every subset of $G_i$ of size at most $\a n$ expands.
Suppose $G_i$ does not have a perfect matching. 
Then $|B(v_i)|= |B_{G_i}(v_i)|> \a n$, where $v_i$ denotes the vertex exposed in the $i$-th step, and at most $i-1$ of the vertices in $B(v_i)$ is larger that $v_i$. 
Hence the probability of success, that is, the probability of extending the maximum matching size is at least 
\beq\label{MSP}
1-\lp 1-\frac{\a n -(i-1)}{n} \rp^k
\eeq
in step $i$. If $G_{\a n}$ does not have a perfect matching, then the number of successes is smaller than $\gamma n/2$ by the time first $\a n$ vertices are exposed.
On the other hand, this probability is bounded above by 
\[
\pr(Y_1+\cdots+Y_{\a n} <\gamma n/2),
\]
where $Y_i$ is a Bernoulli random variable with parameter $1-(1-(\a n-(i-1))/n)^k$ (independent of all the others).
Since
\[
\sum \mean[Y_i] \sim n \int_{0}^\a 1-(1-(\a-x))^k dx = \lp \a -\frac{1}{k+1}+\frac{(1-\a)^{k+1}}{1+k}\rp n=\zeta n,
\]
the Chernoff bound~\eqref{Chernoff} gives $Y_1+\cdots+Y_{\a n}\sim \zeta n$ whp, which finishes the proof.
\end{proof}




We are now ready to prove our main theorems about matchings. Recall that Theorem~\ref{thm:k=4} states that only a bounded number of vertices are isolated in $\ug[n,4]$ whp, and Theorem~\ref{thm:k=5} states that $\ug[n,5]$ has a perfect matching whp.

\begin{proof}[{\bf Proof of Theorem~\ref{thm:k=4}}]
Let $\omega$ be an integer tending to infinity slowly. We want to show that a maximum matching in $\ug[n,4]$ isolates fewer than $\omega$ vertices whp.

In order to get $\ug[n,4]$, we reveal the vertices and their choices one at a time.
Let $G_t$ denote the graph after $t$ vertices are revealed. (So in $G_t$ we only see $4(t-1)$ edges.)
Let  $\kappa_t$ denote the number of vertices that are isolated by a maximum matching in $G_t$.
We need to show
\beq\label{kappa is not large}
\pr(\k_n\ge \omega) \to 0
\eeq
as $n\to \infty$.

Let $t_0=\lfloor \sqrt{n}\rfloor$ and $\a=0.172$, which is smaller than $\a_1(4)$ by~\eqref{a1 values}.
For $t_0\le t\le n$, let ${\cal E}_t$ be the event that $G_t$ is an $(\a,1)$-expander. 
Using Corollary~\ref{a1,a2 expanders} with $k=4$ gives that each $G_t$ is an $(\a,1)$-expander with  probability $1-O(t^{-5})$. 
Hence, defining the event ${\cal E}:=\cap_{t_0\le t\le n} {\cal E}_t$, we have
\[
\pr(\cal E) \ge 1-\sum_{t=t_0}^n \pr(\cal E_t^c) = 1- \sum_{t=t_0}^n O\left(t^{-5}\right)= 1-O\left(n^{-2}\right).
\]
From now on, we will condition on $\cal E$.

Let $A_t$ denote the set of vertices that are isolated by at least one maximum matching in $G_t$.
(Hence $\k_t=0$  if and only if  $A_t=\emptyset$.)
Note that
\[
\k_{t+1}=
\begin{cases}
\k_t-1 & \text{ if $t+1$ chooses a vertex from $A_t$,}\\
\k_t+1 & \text{ otherwise.}
\end{cases}
\]
By Lemma~\ref{B and N(B)}, for any $v\in A_t$, the set $B(v)$ does not expand, that is, $|N(B(v))|<|B(v)|$.
Consequently, $|B(v)|$, and hence $A_t$, has size at least $\a t$ for $t\ge t_0$.
Thus, for $t\ge t_0$,
\beq\label{kappa changes}
\begin{aligned}
\pr(\kappa_{t+1}=\kappa_t+1 | \k_t\not=0)& \le (1-0.172)^4< 0.48\\
\pr(\kappa_{t+1}=\kappa_t-1 | \k_t\not=0) &= 1-\pr(\kappa_{t+1}=\kappa_t+1| \k_t\not=0) \ge 1-(1-0.172)^4 > 0.52.
\end{aligned}
\eeq
A crucial point is that $\k_t$ tends to decrease as long as  the maximum matching is not a perfect matching and we want to use this to prove~\eqref{kappa is not large}.

For $t\ge t_0$, let $\xi_{t}$ be the indicator of the event $\{\k_{t-1}\not=0, \ \k_{t}=\k_{t-1}+1\}$. 
So $\xi_t$ gets the value 1 when vertex $t$ has a chance to improve the maximum matching size but fails to do so.
For $t\ge t_0$, by~\eqref{kappa changes}, the sum $\xi_{t+1}+\cdots+\xi_n$ is stochastically dominated by a binomial random variable $Z_t$ with parameters $n-t$ and $0.48$.
Let $\cal T$ be the set of times $t$ at which a perfect matching occurs, that is,
\[
\Cal T= \{t \in [t_0,n]: \k_t=0\}.
\]
Trivially,
\beq\label{kappa two terms}
\pr(\k_n\ge \omega) \le \pr(\mathcal T=\emptyset)+  \pr(\k_n\ge \omega,\ \mathcal T\not=\emptyset).
\eeq
A requirement for $\cal T=\emptyset$ is 
\[
\sum_{t=t_0+1}^n\xi_t > \frac{n}{2}-t_0.
\]
Hence, by the Chernoff bound in~\eqref{Chernoff} (used only for the last inequality below),
\beq\label{term 1}
\pr(\mathcal T=\emptyset) \le \pr\left(\sum_{t=t_0+1}^n\xi_t > \frac{n}{2}-t_0\right) \le \pr( Z_{t_0} > n/2-t_0) \le e^{-\Omega(n)}.
\eeq
Now let us bound $ \pr(\k_n\ge \omega,\ \mathcal T\not=\emptyset)$.
When $\cal T\not=\emptyset$, let $t_f:=\max\mathcal T$. In this case, since $\k_n\le n-t_f$,
\begin{align*}
\pr(\k_n\ge \omega,\ \mathcal T\not=\emptyset) 
&= \pr(\k_n\ge \omega,\  \mathcal T\not=\emptyset,\ t_f>n-\omega)+ \pr(\k_n\ge \omega,\ \mathcal T\not=\emptyset,\ t_f \le n-\omega)\\
&=\pr(\k_n\ge \omega,\  \mathcal T\not=\emptyset,\ t_f\le n-\omega)	)	\\	
&=\sum_{t=t_0}^{n-\omega} \pr(\k_n\ge \omega,\ \mathcal T\not=\emptyset,\ t_f= t)
\end{align*}
Finally, the event $\{\k_n\ge \omega,\ \mathcal T\not=\emptyset,\ t_f= t\}$ is a subevent of
\[
1+\sum_{t'=t+1}^n\xi_t > \frac{n-t+\omega}2,
\]
from which we get,
\[
\pr(\k_n\ge \omega,\ \mathcal T\not=\emptyset,\ t_f= t) \le \pr\left( \sum_{t'=t+1}^n\xi_t > \frac{n-t}2\right)\le \pr\left(Z_t>\frac{n-t}{2}\right)\le e^{-\Omega(n-t)},
\]
where the last inequality follows from the Chernoff bound~\eqref{Chernoff}.
Hence
\beq\label{term 2}
\pr(\k_n\ge \omega,\ \mathcal T\not=\emptyset) = \sum_{t=t_0}^{n-\omega} \pr(\k_n\ge \omega,\ \mathcal T\not=\emptyset,\ t_f= t) \le \sum_{t=t_0}^{n-\omega}e^{-\Omega(n-t)}\to 0.
\eeq
Using~\eqref{term 1} and \eqref{term 2} in~\eqref{kappa two terms} gives~\eqref{kappa is not large}, which finishes the proof.
\end{proof}

\begin{proof}[{\bf Proof of Theorem~\ref{thm:k=5}}]
We view $\ug[n,5]$ as the union of $G=\ug[n,4]$ and $\ug[n,1]$. 
By Theorem~\ref{thm:k=4}, almost all the vertices of $G$ are covered by a maximum matching. 
Also, for $\a=0.172$, every set of size at most $\a n$ expands in $G=\ug[n,4]$ by~\eqref{a1 values}. 
We apply Lemma~\ref{thm: alpha and gamma} with $k=1$, $\a=0.172$, and a sufficiently small $\ga>0$ to finish the proof. 
(Note that $\zeta=\a- \frac{1}{2}+ \frac{(1-\a^2)}{2}>0$.)
\end{proof}

\section{Hamilton cycles}\label{sec:cycles}

The proofs in this section are analogous to the proofs in the previous section and also follow closely the proofs of the similar results in~\cite{FGPR}. We find a Hamilton cycle again in two stages. In the first stage, for some $k_1<k$, we reveal the vertices (and their $k_1$ choices) one at a time, and find a long path in $\ug[n,k_1]$. In the second stage, we reveal the remaining $k-k_1$ choices of each vertex in some particular order and complete the long path to a Hamilton cycle in $\ug[n,k]$. 
In each stage, the key tool we need is that $\ug[n,k]$ is an $(\a,2)$-expander for every constant $\a<\a_2(k)$.
This fact provides us with many nonedges whose additions to the graph would increase the length of the longest path.
Some approximate values (lower bounds) of $\a_2(k)$'s for small $k$ are given in~\eqref{a2 values}.

Given a longest path $P=x_0x_1\dots x_t$ in a graph $G$, if $x_i \in P$ is a neighbor of $x_t$ , then $P'=x_0\dots x_{i}x_tx_{t-1}\dots x_{i+1}$  is another longest path. This transformation from $P$ to $P'$, which was introduced by P\'osa~\cite{Posa}, is called a rotation. The set $\End(P,x_0)$ is defined as the set of vertices $x$ such that $x_0$ and $x$ are connected through a path $Q$ that is obtained from $P$ through a sequence of rotations. (Vertex $x_0$ stays as an endpoint in all the rotations.) The collection of these sets is crucial for our proofs. 

Analogous to Lemma~\ref{B and N(B)}, we have the following lemma. (See \cite[Corollary 6.7]{FK} for a proof.)

\begin{lemma}\label{N<2X}
Let $G = (V,E)$ be a graph, $P$ be any longest path of $G$, and $a$ one of its endpoints. Then, $|N(\End(P, a))| < 2|\End(P, a)|$.
\end{lemma}

The next lemma is analogous to Lemma~\ref{thm: alpha and gamma}.

\begin{lemma} \label{thm: alpha and lambda}
Let $\lambda$ and $\alpha$ be two positive constants.
Suppose $G$ is a graph on the vertex set $[n]$ that has a path of length at least $\lambda n$ vertices, where $n\to \infty$.
Suppose also that $G$ is an $(\a,2)$-expander.
Let $k$ be a positive  integer such that
\[
\zeta:=\alpha- \frac{1}{k+1}+\frac{(1-\alpha)^{k+1}}{k+1}.
\]
\begin{enumerate}
\item[(i)] If $\zeta> 1-\lambda$, then $G \cup \ug$ has a Hamilton cycle whp.
\item[(ii)]If $\zeta< 1-\lambda$, then $G \cup \ug$ has a path of length at least $(1-o(1))(\lambda+\zeta)$ whp.
\end{enumerate}
\end{lemma}

\begin{proof}
The proof is very similar to the proof of Lemma~\ref{thm: alpha and gamma} and we omit some of the details. 
We will use the following standard argument in the proof of the lemma. 
\begin{claim}
Let $Q$ be a longest path with endpoints $a$ and $b$ in a connected non-Hamiltonian graph $H$. Then the graph $H\cup \{ab\}$ is either Hamiltonian or has a path longer than $Q$.
\end{claim}

\begin{proof}[Proof of the claim]
The edge $ab$ together with $Q$ forms a cycle $C$. If $C$ is not a Hamilton cycle, then, since $G$ is connected, there must be an edge $cd$ in $H$ such that $c\in C$ and $d\not \in C$. In that case, $C \cup \{cd\} \setminus \{cc'\}$, where $cc'$ is an edge in $C$, is a path longer than $Q$. 
\end{proof}

It follows immmediately from the claim that if $H$ is a connected non-Hamiltonian graph and $P$ is a longest path with one endpoint $a$, then adding an edge between $a$ and $\End(P,a)$ either makes the graph Hamiltonian or gives a path longer than $P$. In our case, since $G$ given in the lemma is connected and an $(\a,2)$-expander, by Lemma~\ref{N<2X}, $\End(P,a)$ is large for every pair $(P,a)$.
Now we describe how we use this fact to improve the longest path sufficiently many times so that the statement of the lemma holds.

We start with the definitions of two sets $A$ and $B$ analogous to the ones in the previous section.
For a given graph $H$, let
\[
A=A(H):=\{v\in V(H)\, :\, v \text{ is an endpoint of some longest path in $H$}\},
\]
and for $v\in A$,
\[
B(v):=\{w\in V(H)\setminus \{v\} \,:\, w \text{ is the other endpoint of a longest path in $H$ that starts at $v$}\}.
\]
Let $G_0=G$. For $i\ge 0$, recursively, we obtain $G_{i+1}$ from $G_i$ by exposing a particular vertex in $\gnk$. 
Specifically, denoting the  largest unexposed vertex in $A(G_i)$ by $w_i$, we let $G_{i+1}:=G_i\cup E_i$, where $E_i$ is the set of edges in $\gnk$ that connects $w_i$ with older vertices $\{1,\dots,w_i-1\}$.

Now, by the previous discussion, if $w_i$ chooses a vertex from $B(w_i)=B_{G_i}(w_i)$, then a longest path is improved (or a Hamilton cycle is obtained).
For every $i$, since $G_0\subseteq G_i$ and $G=G_0$ is an $(\a,2)$-expander, $G_i$ is also an $(\a,2)$ expander. Hence, by Lemma~\ref{N<2X}, $|B(w_i)|\ge \a n$ for every $i$.
Hence, the probability that $w_i$ chooses at least one vertex from $B(w_i)$ in $\gnk$, or equivalently, the probability of extending the longest path, is
\[
1-\brac{1-\bfrac{B(w_i)-i}{n}}^k \ge 1-\brac{1-\bfrac{\a n-i}{n}}^k.
\]
The rest of the proof is similar to what follows Equation~\eqref{MSP} and we will not repeat it.
\end{proof}

We now prove Theorem~\ref{thm:k=13} which tells us that $\ug[n,13]$ is Hamiltonian.

\begin{proof}[{\bf Proof of Theorem~\ref{thm:k=13}}]
Let $H_0=\ug[n,10]$, and $H_i$ be the graph obtained from $H_{i-1}$ by adding an independent copy of $\ug[n,1]$.
Hence, $H_i$ has the same distribution as $\ug[n,10+i]$.
We will first show that $H_0$ has a long path and then use Lemma~\ref{thm: alpha and lambda} repeatedly to show that $H_3$ is Hamiltonian. 

\begin{claim}
Whp, $H_0$ has a path of length at least $(0.9177)n$.
\end{claim} 

\begin{proof}[Proof of the claim]
As the proof is very similar to the proof of Theorem~\ref{thm:k=4}, we will leave some details out.
We reveal the vertices and their $10$ choices one at a time, and we find a path of desired length.
For each $t$, after step $t$, we pick a longest path $P_t$ and one of its endpoints $a_t$ (with any deterministic rule or randomly). 
By Lemma~\ref{N<2X}, we have $|N(\End(P_t, a_t))| < 2|\End(P_t, a_t)|$.
On the other hand, by Corollary~\ref{a1,a2 expanders} and Equation~\eqref{a2 values}, $G_t$ is an $(0.221,2)$-expander with probability $1-O(t^{-8})$.
Hence, for $t_0:=\floor{\sqrt n}$,
\[
\mathcal Q_t:= \big\{|\End(P_t, a_t)|\ge (0.221) t\big\},
\]
and $\mathcal Q:=\cap_{t\ge t_0} \ \mathcal Q_t$, we have 
\[
\pr(\mathcal Q) \ge 1-\sum_{t\ge t_0} O\left(t^{-8}\right)=1-O\left(t_0^{-7}\right)=1-O\left(n^{-7/2}\right).
\]
Let $\xi_t$ denote the indicator of the event $\{\text{vertex $t+1$ makes a choice from  } \End(P_t, a_t)\}$, which is contained in $\{|P_{t+1}|\ge 1+ |P_t|\}$.
As $n\to \infty$ and $t\ge t_0= \floor{\sqrt n}$,
\[
\mean[\xi_t] \ge \pr(\mathcal Q)\cdot\mean[\xi_t\, |\, \mathcal Q] \ge \left(1-O\left(n^{-7/2}\right)\right)\left(1- (1-0.221)^{10}\right)> 0.9177.
\]
Hence $\sum_{t=t_0}^n \xi_t$ stochastically dominates the sum of $n-t_0+1$ indepedent identical Bernoulli random variables with mean $0.9177+\eps$ for some small but fixed  $\eps>0$. Hence by the law of large numbers,
\[
\sum_{t=t_0}^n \xi_t > (0.9177)n.
\]
This finishes the proof of the claim.
\end{proof}
Note that $H_i$ is an $(\a,2)$-expander whp for any $\a<\a_2(i+10)$, where 
\[
\a_2(10)>0.221, \quad \a_2(11)>0.235, \quad \a_2(12)>0.247
\]
as given in~\eqref{a2 values}.
In particular, whp, $H_0$, $H_1$, and $H_2$ are $(0.221,2)$, $(0.235,2)$, and $(0.247,2)$-expanders, respectively.
Using Lemma~\ref{thm: alpha and lambda} with $G=H_0$, $k=1$, $\la=0.9177$, and $\a=0.221$ gives the existence in $H_1$ of a path of length at least
\[
(1-o(1)) (0.9177+0.221^2/2)n > (0.9421)n
\]
whp. Similarly, whp, $H_2$ has a path of length at least
\[
(1-o(1)) (0.9421+0.235^2/2)n > (0.9697)n.
\]
Now since $\a_2(12)^2/2>(0.247)^2/2>0.0305>1-0.9697$, again by Lemma~\ref{thm: alpha and lambda}, $H_3$ has a Hamilton cycle whp.
\end{proof}

\section*{Acknowledgement}

The author is grateful to Boris Pittel for his valuable suggestions.

\end{document}